\documentclass[a4paper,11pt]{amsart}

\usepackage{comment}

\usepackage{etoolbox}
\ifdef{\crop}{%
\usepackage[includeheadfoot,twoside=False,paperwidth=448pt,paperheight=587pt,rmargin=15pt,lmargin=15pt,tmargin=15pt,bmargin=15pt]{geometry}%
}{%
\setlength{\topmargin}{22mm}
\addtolength{\topmargin}{-1in}
\setlength{\oddsidemargin}{27mm}
\addtolength{\oddsidemargin}{-1in}
\setlength{\evensidemargin}{27mm}
\addtolength{\evensidemargin}{-1in}
\setlength{\textwidth}{156mm}
\setlength{\textheight}{240mm}
}%

\usepackage{here}
\usepackage{color}
\usepackage[active]{srcltx}
\usepackage{amsmath,amsthm,amsxtra}
\usepackage{amssymb}

\usepackage[mathscr]{eucal}

\usepackage{bm}
\usepackage[all]{xy}

\usepackage{aliascnt}

\usepackage{tabularx}
\newcolumntype{C}{>{\centering\arraybackslash}X} 

\theoremstyle{plain}
\newtheorem{thm}{Theorem}[section]
\newtheorem*{thm*}{Theorem}

\newaliascnt{prop}{thm}
\newaliascnt{cor}{thm}
\newaliascnt{lem}{thm}
\newaliascnt{claim}{thm}
\newaliascnt{defn}{thm}
\newaliascnt{ques}{thm}
\newaliascnt{conj}{thm}
\newaliascnt{fact}{thm}
\newaliascnt{rem}{thm}
\newaliascnt{ex}{thm}
\newtheorem{prop}[prop]{Proposition}
\newtheorem{cor}[cor]{Corollary}
\newtheorem{lem}[lem]{Lemma}
\newtheorem{claim}[claim]{Claim}
\newtheorem*{prop*}{Proposition}
\newtheorem*{cor*}{Corollary}
\newtheorem*{lem*}{Lemma}
\newtheorem*{claim*}{Claim}
\theoremstyle{definition}

\newtheorem{conj}[conj]{Conjecture}
\newtheorem*{defn*}{Definition}
\newtheorem*{ques*}{Question}
\newtheorem*{conj*}{Conjecture}

\newtheorem*{prob*}{Problem}

\newtheorem{rem}[rem]{Remark}
\newtheorem{ex}[ex]{Example}
\newtheorem*{fact*}{Fact}
\newtheorem*{rem*}{Remark}
\newtheorem*{ex*}{Example}
\aliascntresetthe{prop}
\aliascntresetthe{cor}
\aliascntresetthe{lem}
\aliascntresetthe{claim}
\aliascntresetthe{defn}
\aliascntresetthe{ques}
\aliascntresetthe{conj}
\aliascntresetthe{fact}
\aliascntresetthe{rem}
\aliascntresetthe{ex}

\usepackage[pointedenum]{paralist}
\usepackage{varioref}
\labelformat{equation}{\textnormal{(#1)}}
\labelformat{enumi}{\textnormal{(#1)}}

\usepackage[a4paper]{hyperref}

\def\textsectionN~{\textsection{}}

\renewcommand\phi{\varphi}
\renewcommand\epsilon{\varepsilon}
\renewcommand\leq{\leqslant}
\renewcommand\geq{\geqslant}

\makeatletter
\newcommand{\set}{  \@ifstar{\@setstar}{\@set}}\newcommand{\@setstar}[2]{\{\, #1 \mid #2 \,\}}
\newcommand{\@set}[1]{\{ #1 \}}
\makeatother

\newcommand{\trans}[1][1]{\raisebox{#1ex}{\scriptsize\kern0.1em$t$\kern-0.1em}}

\DeclareMathOperator{\codim}{codim}
\DeclareMathOperator{\Hom}{Hom}

\DeclareMathOperator{\Pic}{Pic}

\DeclareMathOperator{\id}{id}

\DeclareMathOperator{\Sym}{Sym}

\DeclareMathOperator{\pr}{pr}

\def\Z{\mathbb{Z}}
\def\Q{\mathbb{Q}}
\def\R{\mathbb{R}}
\def\C{\mathbb{C}}

\def\r+{\mathbb{R}_{\geq 0}}

\def\ep{\varepsilon}

\def\r+{{\R}_{\geq 0}}
\def\q+{{\Q}_{\geq 0}}
\def\P{\mathbb{P}}

\def\*c{\C^{\times}}

\def\<{\langle}
\def\>{\rangle}

\def\ni{\noindent}
\def\ra{\rightarrow}
\def\lra{\Leftrightarrow}

\def\C{\mathbb {C}}

\def\Q{\mathbb {Q}}
\def\R{\mathbb {R}}

\def\Z{\mathbb {Z}}

\newcommand{\calf}{\mathcal {F}}

\newcommand{\cali}{\mathcal {I}}

\newcommand{\call}{\mathcal {L}}

\newcommand{\calo}{\mathcal {O}}

\newcommand{\calx}{\mathcal {X}}

\makeatletter
  
  \@addtoreset{equation}{section}
\makeatother

\title[Higher syzygies on general polarized abelian varieties of type $(1,\dots,1,d)$]{Higher syzygies on general polarized abelian varieties of type $(1,\dots,1,d)$}

\author[A.~Ito]{Atsushi~Ito}
\address{Graduate School of Mathematics,
Nagoya University,
Nagoya, Japan}
\email{atsushi.ito@math.nagoya-u.ac.jp}


\begin{document}

\maketitle

\begin{abstract}
In this paper,
we show that 
a general polarized abelian variety $(X,L)$ of type $(1,\dots,1,d)$ and dimension $g$
satisfies property $(N_p)$ if
$ d \geq \sum_{i=0}^{g} (p+2)^i$.
In particular,
the case $p=0$ affirmatively solves a conjecture by L.~Fuentes Garc\'{\i}a on projective normality. 
\end{abstract}

\section{Introduction}

Throughout this paper, we work over the complex number field $\C$.

For an ample line bundle $L$ on an abelian variety $X$ of dimension $g$,
we can associate a sequence of positive integers $(d_1,\dots,d_g)$ with $d_1|d_2|\cdots|d_g$,
called the \emph{type} of $L$.
It is well known that $L$ is basepoint free if $d_1 \geq 2$ and projectively normal if $d_1 \geq 3$.
On the other hand,
in the case $d_1=1$,
equivalently the case when $L$ is not written as some multiple of another line bundle, 
basepoint freeness or projective normality of $L$ is more subtle.
In \cite{MR1299059},
the authors investigate general $L$ of type $(1,\dots,1,d)$
and prove the following theorem.

\begin{thm}[{\cite[Proposition 2, Proposition 6, Corollary 25]{MR1299059}}]\label{thm_DHS}
Let $(X,L)$ be a general polarized abelian variety of type $(1,\dots,1,d)$ and dimension $g$.
Then 
\begin{enumerate}
\item $L$ is base point free if and only if $d \geq g+1$.
\item The morphism defined by $|L|$ is birational onto the image if and only if $d \geq g+2$.
\item $L$ is very ample if $d > 2^g$.
\end{enumerate}
\end{thm}

On the other hand, L.~Fuentes Garc\'{\i}a investigates projective normality based on the work \cite{MR1974682} of J.~N.~Iyer,
and conjectures the following:

\begin{conj}[{\cite[Conjecture 4.7]{MR2181770}}]\label{conj_FG}
Let $(X,L)$ be a general polarized abelian variety of type $(1,\dots,1,d)$ and dimension $g$.
Then $L$ is projectively normal if $d \geq 2^{g+1} -1$.
\end{conj}

For $g=2$,
this conjecture follows from \cite{Lazarsfeld_proj_normal}, \cite{MR2076454} or \cite{MR1667600}.
Fuentes Garc\'{\i}a proves this conjecture for $g=3,4$
using results in \cite{MR1974682} and some calculations of the ranks of suitable matrices
using computer.
We note that $d=h^0(L) \geq 2^{g+1} - 1 $ is a necessary condition for the projective normality of $L$
since $\dim \Sym^2 H^0(L) \geq \dim H^0(L^{\otimes 2})$ must hold for such $L$.
Hence \autoref{conj_FG} states that it is a sufficient condition as well for general $(X,L)$.

In this paper,
we prove \autoref{conj_FG} affirmatively.
In fact, we prove not only projective normality but also property $(N_p)$ as follows:

\begin{thm}\label{thm_intro}
Let $p \geq -1$ be an integer and
let $(X,L)$ be a general polarized abelian variety of type $(1,\dots,1,d)$ and dimension $g$.
Then $L$ satisfies property $(N_p)$ if
\[
d \geq \sum_{i=0}^{g} (p+2)^i = 
  \begin{cases}
    g+1 & \text{ if } \ p=-1 \\
   ( (p+2)^{g+1}-1)/(p+1)
    & \text{ if } \ p \geq 0.
  \end{cases}
\]

In particular,
$L$ is projectively normal if and only if $d \geq 2^{g+1} -1$.
In this case, the homogeneous ideal of $X$ embedded by $|L|$ is generated by quadrics and cubics.  
\end{thm}

We refer the readers to \cite[Chapter 1.8.D]{MR2095471}, \cite{MR2103875} for the definition of property $(N_p)$.
We just note here that 
$(N_p)$'s consist an increasing sequence of positivity properties.
For example, 
($N_0$) holds for $L$ if and only if $L$ defines a projectively normal embedding,
and ($N_1$) holds if and only if ($N_0$) holds and the homogeneous ideal of the embedding is generated by quadrics. 
Usually $(N_p)$ is considered for $p \geq 0$,
but we add the basepoint freeness in the sequence of positivity properties as $(N_{-1})$, as in \cite{lozovanu:2018}, \cite{jiang2021}.
Hence the case $p=-1$ of \autoref{thm_intro} recovers \autoref{thm_DHS} (1)
and the case $p=0$ of \autoref{thm_intro} proves \autoref{conj_FG} affirmatively.

For an abelian surface $X$, 
it is known that a very ample line bundle $L$ of type $(1,d)$ with $d \geq 7$ is projectively normal and the homogeneous ideal of $X$ embedded by $|L|$
is generated by quadrics and cubics by 
\cite{Lazarsfeld_proj_normal}, \cite{MR2076454}, \cite{MR3656291}.
The last statement of \autoref{thm_intro} generalizes this result to higher dimensions at least for general $(X,L)$.
Furthermore,
\autoref{thm_intro} gives a better bound of $d$ than the bounds obtained by \cite{MR4009173}, \cite{MR3923760} in dimension two 
and \cite{MR1974682}, \cite{MR2833789}, \cite{Ito:2020aa}, \cite{jiang2021} in higher dimensions.
See Remarks \ref{rem_surface}, \ref{rem_related_results} for details.

\vspace{2mm}
In the rest of Introduction,
we explain the idea of the proof.
In \cite{MR4157109}, Z.~Jiang and G.~Pareschi introduce
an invariant $\beta(l)=\beta(X,l) \in (0,1]$,
called the \emph{basepoint-freeness threshold}, 
for the class $l \in \Pic X/\Pic^0 X$ of an ample line bundle $L$.
By \cite{MR4157109}, \cite{MR4114062}, 
a suitable upper bound of $\beta(l)$ implies property $(N_p)$ as follows:

\begin{thm}[{\cite[Theorem D, Corollary E]{MR4157109},\cite[Theorem 1.1]{MR4114062}}]\label{thm_ Bpf_threshold}
Let $p \geq -1$ and
let $(X,L)$ be a polarized abelian variety. Then
$L$ satisfies $(N_p)$ if $\beta(l)  < 1/(p+2)$.
\end{thm}

Hence the following theorem implies \autoref{thm_intro}.

\begin{thm}\label{main_thm}
Let $d ,g \geq 1$ be integers and set $m:=\lfloor  \sqrt[g]{d} \rfloor$.
Let $(X,l)$ be a general polarized abelian variety of type $(1,\dots,1,d)$ and dimension $g$.
Then 
\begin{enumerate}
\item $1/\sqrt[g]{d} \leq \beta(l) \leq 1/m$ holds.
\item $1/\sqrt[g]{d} \leq \beta(l)  < 1/m$ holds if $d \geq m^{g} + \cdots + m+1 = (m^{g+1}-1)/(m-1)$.
\end{enumerate}
\end{thm}

\begin{rem}
Upper bounds of $\beta(l)$ imply not only properties $(N_p)$ but also jet ampleness and
vanishings of suitable Koszul cohomologies: 
\cite[Theorem D]{CaucciThesis}, \cite[Proposition 2.5]{Ito:2020aa} and \autoref{main_thm} imply that
$L$ is $p+1$-jet ample and
the Koszul cohomology group $K_{p,q}(X,L;kL)=0$ for any $q,k \geq 1$
if $(X,L)$ is a general polarized abelian variety of type $(1,\dots,1,d)$ and dimension $g$ with $d \geq \sum_{i=0}^{g} (p+2)^i$.
\end{rem}

In \cite{Ito:2020aa},
the author observes a similarity between $\beta(l)^{-1}$ and Seshadri constants.
Since Seshadri constants are lower-semicontinuous,
it is natural to ask whether $\beta(l)$ is upper-semicontinuous or not.
In fact, this is the case as we see in \autoref{sec_semicontinuity}.
Hence,
\autoref{main_thm} is reduced to finding 
an example $(X_0,L_0)$ of type $ (1,\dots,1,d)$ 
such that $\beta(l_0) \leq  1/m $ or $< 1/m$.

\autoref{thm_DHS} (3) is proved in \cite{MR1299059} 
by degenerating polarized abelian varieties to a polarized variety $(X'_0,L'_0)$
whose normalization is a $(\P^{1})^{g-1}$-bundle over an elliptic curve
and showing that $L'_0$ is very ample.
Contrary to very ampleness,
$\beta(l) $ is defined only for abelian varieties.
Hence we do not use such degenerations 
but find $(X_0,L_0)$ as a polarized abelian variety.
In fact,
we construct such $(X_0,L_0) $ as a suitable polarization on a product of elliptic curves.

We note that
\autoref{thm_ Bpf_threshold} is also used to show $(N_p)$ 
in \cite{Ito:2020aa}, \cite{jiang2021},
where techniques to cutting minimal log canonical centers are used
to bound $\beta(l)$ from above.
In this paper, we do not need such techniques.

\vspace{2mm}
This paper is organized as follows. 
In \autoref{sec_preliminary}, we recall some notation.
In \autoref{sec_semicontinuity},
we show the upper-semicontinuity of $\beta(l)$.
In \autoref{sec_surface},
we study $\beta(l)$ of polarized abelian surfaces 
and show \autoref{main_thm} for $g=2$.
In \autoref{sec_any_dim},
we prove Theorems \ref{thm_intro}, \ref{main_thm} in any dimension.
In 
Appendix,
we compute $\beta(l)$ of general polarized abelian surfaces $(X,l)$ of type $(1,d)$ for some $d$.

\subsection*{Acknowledgments}
The author would like to express his gratitude to
Professor Zhi Jiang for sending drafts of \cite{jiang2021} to the author.
He also thanks Professor Victor Lozovanu for valuable comments.
The author was supported by JSPS KAKENHI Grant Number 17K14162, 21K03201.

\section{Preliminaries}\label{sec_preliminary}

Let $X$ be an abelian variety of dimension $g$.
We denote the origin of $X$ by $o_X$ or $o \in X$.
For $b \in \Z$, we denote the multiplication-by-$b$ isogeny by
\[
\mu_b=\mu^X_b : X \rightarrow X , \quad p \mapsto bp.
\]

For an ample line bundle $L$ on $X$,
we call $(X,L)$ or $(X,l)$ a \emph{polarized abelian variety},
where $l \in \mathrm{NS}(X)=\Pic(X)/\Pic^0(X)$ is the class of $L$ in the Neron-Severi group of $X$.
Let 
\[
K(L)=\{ p \in X \, | \, t_p^* L \simeq L  \},
\]
where $t_p : X \rightarrow X$ is the translation by $p$ on $X$.
It is known that there exist positive integers $d_1|d_2 | \cdots|d_g$ such that 
$K(L) \simeq (\bigoplus_{i=1}^g \Z/d_i\Z )^{\oplus 2}$ as abelian groups.
The vector $(d_1,\dots,d_g)$ is called the \emph{type} of $L$.
Since $K(L)$ depends only on the class $l$ of $L$,
 $(d_1,\dots,d_g)$ is called the type of $l$ as well.
 It is known that $\chi(l)  =  \prod_{i=1}^g d_i$ holds.

\vspace{2mm}
For a coherent sheaf $\calf$ on $X$ and $x \in \Q$,
a \emph{$\Q$-twisted coherent sheaf} $\calf \langle xl \rangle$ is the equivalence class of the pair $(\calf,xl)$,
where the equivalence is defined by
\[
(\calf \otimes L^{m} , xl) \sim (\calf  , (x+m)l)
\]
for any line bundle $L$ representing $l$ and any $m \in \Z$.

Recall some notions of generic vanishing:
a coherent sheaf $\calf$ on $X$ is said to be \emph{IT(0)} if $h^i(X,\calf \otimes P_{\alpha}) =0$ for any $i >0$ and any $\alpha \in  \widehat{X}=\Pic^0 (X)$,
where $P_{\alpha}$ is the algebraically trivial line bundle on $X$ corresponding to $\alpha$.
It is said to be \emph{GV} if 
 $\codim_{\widehat{X}} \{ \alpha \in \widehat{X} \, | \,  h^i(X,\calf \otimes P_{\alpha}) >0 \} \geq i$ for any $i >0$.

In \cite{MR4157109}, 
such notions are extended to the $\Q$-twisted setting.
A $\Q$-twisted coherent sheaf $\calf \langle xl \rangle$ for $x=a/b$ with $b >0$ is said to be IT(0) or GV
if so is $\mu_b^* \calf \otimes L^{ab}$.
We note that this definition does not depend on the representation $x=a/b$ nor the choice of $L$ representing $l$.
By \cite[Theorem 5.2]{MR4157109},
$\calf \langle xl \rangle$ is GV if and only if
$\calf \langle (x+x')l \rangle$ is IT(0) for any rational number $x' >0$.

In \cite{MR4157109}, 
an invariant $0 < \beta(l) \leq 1$
is introduced for a polarized abelian variety $(X,l)$.
It is defined using cohomological rank functions, which are also defined in \cite{MR4157109}, but
$\beta(l)$ is characterized by the notion IT(0) as follows:

\begin{lem}[{\cite[Section 8]{MR4157109},\cite[Lemma 3.3]{MR4114062}}]\label{lem_beta_IT(0)}
Let $(X,l)$ be a polarized abelian variety and $x \in \Q$.
Then $\beta(l) < x$ if and only if $\cali_p \langle xl \rangle$ is IT(0) for some (and hence for any) $p \in X$,
where $\cali_p \subset \calo_X$ is the ideal sheaf corresponding to $p \in X$.
\end{lem}

\begin{rem}\label{rem_GV}
For a rational number $x =a/b >0$,
$\beta(l) \leq x$ if and only if $\cali_p \langle xl \rangle$ is GV for some (and hence for any) $p \in X$
since
\begin{align*}
\beta(l) \leq x \ &\Longleftrightarrow \ \beta(l) < x + x' \text{ for any rational number } x' > 0\\
&\Longleftrightarrow \ \cali_p \langle (x+x')l \rangle  \text{ is IT(0) for any rational number } x' >0\\
&\Longleftrightarrow  \ \cali_p \langle xl \rangle  \text{ is GV}.
\end{align*}
Fix a representative $L$ of $l$.
By the exact sequence 
\[
 0 \ra \mu_b^* \cali_p \otimes L^{ab} \otimes P_{\alpha} \ra L^{ab} \otimes P_{\alpha} \ra  (\calo_X/\mu_b^* \cali_p) \otimes  L^{ab} \otimes P_{\alpha}  \ra 0
\]
and $ h^i(L^{ab} \otimes P_{\alpha})=h^i(( \calo_X/\mu_b^* \cali_p )\otimes  L^{ab} \otimes P_{\alpha} ) =0$ for $i \geq 1$, 
we have $h^i(\mu_b^* \cali_p \otimes L^{ab} \otimes P_{\alpha}) =0$ for any $\alpha \in \widehat{X}$ and $i \geq 2$.
Hence
$\cali_p \langle xl \rangle$ is GV if and only if $h^1(\mu_b^* \cali_p \otimes L^{ab} \otimes P_{\alpha}) =0$ for some $\alpha$.
Equivalently,
$\cali_p \langle xl \rangle$ is GV if and only if $h^1(\mu_b^* \cali_{p'} \otimes L^{ab} ) =0$ for some $p' \in X$.
\end{rem}

We use the following lemma to estimate $\beta(l)$. 

\begin{lem}[{\cite[Lemmas 3.4, 4.3]{Ito:2020aa}}]\label{lem_divisor}
Let $(X,l)$ be a polarized abelian $g$-fold.
Then 
\begin{enumerate}
\item[(i)] $\beta(l) \geq 1/ \sqrt[g]{\chi(l)}  $. 
\item[(ii)] For an abelian subvariety $Z \subset X$,
it holds that $\beta(l) \geq \beta(l|_Z)$.
Furthermore, 
\[
\beta(l|_Z) \leq \beta(l) \leq \max \left\{ \beta(l|_Z), \frac{g (l^{g-1}.Z)}{(l^g)} \right\} = \max \left\{ \beta(l|_Z), \frac{\chi(l|_Z)}{\chi(l)} \right\}  
\]
holds if the codimension of $Z \subset X$ is one.
\end{enumerate}
\end{lem}

\section{Semicontinuity of basepoint-freeness thresholds}\label{sec_semicontinuity}

In this section,
we prove the upper-semicontinuity of $\beta(l)$ as follows:

\begin{thm}\label{thm_semicontinuity}
Let $f : \calx \rightarrow T$ be an abelian scheme over a variety $T$
and let $\call $ be a line bundle on $\calx$ which is ample over $T$.
Set $X_t:=f^{-1}(t), L_t :=\call|_{X_t}$ for $t \in T$
and let $l_t$ be the class of $L_t$.
Take a point $0 \in T$ and a rational number $x$ such that $\beta(X_0,l_0) \leq x$.
Then $\beta(X_t,l_t) \leq x $ holds for general $t \in T$.

In particular, the function $T \rightarrow \R : t \mapsto \beta(X_t,l_t)$ is upper-semicontinuous in Zariski topology.
\end{thm}

\begin{proof}
Note that $ x$ is positive since $0 < \beta(X_0,l_0) \leq x$.
By \autoref{rem_GV}, 
there exists $p_0 \in X_{0}$ such that $h^1(\mu_b^* \cali_{p_0} \otimes L_{0}^{ab} ) =0$,
where $x=a/b$.
By taking a suitable base change of  $\calx \rightarrow T$ by a finite cover $T' \rightarrow T$,
we may assume that there exists a section $ \sigma : T \rightarrow \calx$ such that $\sigma (0) =p_0$.
Then $h^1(\mu_b^* \cali_{\sigma(t)} \otimes L_t^{ab} ) =0$ for general $t \in T$
by the semicontinuity of cohomology.
Hence $\beta(X_t,l_t) \leq x $ holds for general $t \in T$ by \autoref{rem_GV} again.

By definition,
the upper-semicontinuity of $t \mapsto \beta(X_t,l_t)$ is equivalent to the openness of 
$\{ t \in T \, | \, \beta(X_t,l_t) < y   \} $ in Zariski topology for any $y \in \R$.
We already show that $\{ t \in T \, | \, \beta(X_t,l_t) \leq x   \} $ is open for any $x \in \Q$.
Hence 
\[
\{ t \in T \, | \, \beta(X_t,l_t) < y   \}  = \bigcup_{x \in \Q, x < y} \{ t \in T \, | \, \beta(X_t,l_t) \leq x   \} 
\]
is open as well.
\end{proof}

\begin{rem}
By \autoref{thm_semicontinuity},
$\beta(X_t,l_t) \leq  \beta(X_0,l_0)$ holds for general $t \in T$ if $\beta(X_0,l_0)$ is rational.
If $\beta(X_0,l_0)$ is irrational,
though we do not know such examples yet,
we can only say that $\beta(X_t,l_t) \leq  \beta(X_0,l_0)$ holds for \emph{very general} $t \in T$.
\end{rem}

\section{On general polarized abelian surfaces of type $(1,d)$}\label{sec_surface}

Let $k \geq 1$ be an integer and 
take an isogeny $f : E \ra E'$ between elliptic curves with $\ker f \simeq \Z/k \Z$,
e.g.\ we take $ f : \C /( \Z + k\tau \Z) \rightarrow  \C / ( \Z + \tau \Z)$ induced from $\Z +k \tau \Z \subset  \Z +  \tau \Z$.
For the dual isogeny $\hat{f} : E'=\widehat{E'} \ra \widehat{E}=E$, it is well-known that
$
\hat{f} \circ f=\mu^E_{k}$ and $ f \circ \hat{f}=\mu^{E'}_{k} 
$
hold (see e.g.\ \cite[Chapter III, Theorem 6.2]{MR2514094}).

Set $X=E \times E'$ and 
\[
F_1=\{o_E\} \times E', \quad F_2=E \times \{o_{E'}\}, \quad \Gamma = \{ (p, f(p)) \subset E \times E' \, | \, p \in E \}.
\]

\begin{lem}\label{lem_example_surface}
Let $ a,b \geq 0$ be integers such that $(a,b) \neq (0,0)$
and let $l=l_{a,b} $ be the class of $D=aF_1 +bF_2+\Gamma$. 
Then 
\begin{enumerate}[(i)]
\item $(l.F_1)=1+b$.
\item $l$ is a polarization of type $(1,a+ab+bk )$.
\item $1/(1+b) \leq \beta(l) \leq \max\{ 1/(1+b), (1+b)/(a+ab+bk)\}$.
\item $K(l) = \{ (p,q) \in E \times E' \, | \, \hat{f}(q) =(a+k)p , f(p)=(1+b)q \}$.
\end{enumerate}
\end{lem}

\begin{proof}
(i) It is easy to check $(F_1^2)=(F_2^2)=(\Gamma^2)=0$ and $(F_1.F_2)=(F_1.\Gamma) =1, (F_2.\Gamma)=k$.
Hence $(l.F_1)=1+b$ holds.

\vspace{1mm}
\noindent
(ii) 
Since $(l^2) =2(a+ab+bk)> 0 $ and $(l. F_1+F_2)=a+b >0$,
$l$ is ample by \cite[Corollary 4.3.3]{MR2062673}.
Since $\ker f \simeq \Z/k$, $f \in \Hom(E,E')$ is primitive,
that is, $f$ is not written as $c \lambda$ for some integer $c \geq 2$ and some $\lambda \in \Hom(E,E')$.
Hence $l \in \mathrm{NS}(X)$ is primitive as well by 
\cite[Theorem 11.5.1]{MR2062673}.
Since $2 \chi(l) =(l^2) =2(a+ab+bk)$,
the type of $l$ is $(1,a+ab+bk)$.

\vspace{1mm}
\noindent
(iii) holds by applying \autoref{lem_divisor} (ii) to $S=F_1$ since $\beta(l|_{F_1}) =1/\deg (l|_{F_1})=1/(1+b)$.

\vspace{1mm}
\noindent
(iv) By $E\stackrel{\sim}{\rightarrow} \widehat{E} : p \mapsto \calo_{E}(p -o_{E})$ and $E'\stackrel{\sim}{\rightarrow} \widehat{E'} : q \mapsto \calo_{E'}(q -o_{E'}) $,
we may identify
$\widehat{X} =\widehat{E} \times \widehat{E'} $ with $X$.
For $(p,q) \in X$,
the point in $\widehat{X}=X$ corresponding to $t_{(p,q)}^* \calo_X(F_1) \otimes \calo_X(-F_1) = \pr_1^*\calo_{E}((-p)-o_E)$ is $(-p,o_{E'})$,
where $\pr_1 : X \rightarrow E$ is the natural projection.
Similarly, $t_{(p,q)}^* \calo_X(F_2) \otimes \calo_X(-F_2)$ corresponds to $(o_E,-q)$.

\begin{claim}\label{claim_Gamma}
$t_{(p,q)}^* \calo_X(\Gamma) \otimes \calo_X(-\Gamma)$ corresponds to $( \hat{f}(q) - kp, f(p)-q) \in E \times E'$.
\end{claim}

\begin{proof}[Proof of \autoref{claim_Gamma}]
For a line bundle or a divisor  $N$ on $X$,
we can define a group homomorphism 
\[
\varphi_N : X \ra \widehat{X} =X \quad : \quad   x \mapsto t_x^* N \otimes N^{-1}.
\]
We need to show $\varphi_{\Gamma} (p,q) = ( \hat{f}(q) - kp, f(p)-q) $ for $(p,q) \in X=E \times E'$.

Let $\Delta' \subset E' \times E'$ be the diagonal.
For $(q_1,q_2) \in E' \times E'$, we have
\[
t_{(q_1,q_2)}^* \Delta' \cap (E' \times \{o_{E'}\} ) =\{q_2 -q_1\}, \quad \Delta' \cap (E' \times \{o_{E'}\} ) =\{o_{E'}\}
\]
under the identification $E'= E' \times \{o_{E'}\}$.
Hence the algebraically trivial line bundle 
$t_{(q_1,q_2)}^* \calo (\Delta') \otimes \calo(-\Delta') |_{E' \times \{o_{E'}\} }$ on $E'$ corresponds to $ q_2-q_1 \in E'=\widehat{E'}$.
Similarly,
$t_{(q_1,q_2)}^* \calo (\Delta') \otimes \calo(-\Delta') |_{ \{o_{E'}\} \times E' }$ on $E'$ corresponds to $ q_1-q_2 \in E'=\widehat{E'}$.
Thus the map 
\[
\varphi_{\Delta'} : E' \times E' \ra \widehat{E' \times E'} \quad : \quad (q_1,q_2) \mapsto  t_{(q_1,q_2)}^* \calo (\Delta') \otimes \calo(-\Delta')
\]
is written as $\varphi_{\Delta'} (q_1,q_2) = (q_2-q_1,q_1-q_2) \in E' \times E' =\widehat{E' \times E'}  $.

Since $\Gamma$ is the pullback of $ \Delta'$ by $(f,\id_{E'}) : E \times E' \ra E' \times E'$,
we have a commutative diagram
\[
\xymatrix{
E \times E' \ar[r]^-{\varphi_{\Gamma}} \ar[d]_-{(f,\id_{E'})} & \widehat{E \times E'} = E\times E' \\
E' \times E' \ar[r]^-{\varphi_{\Delta'}} & \widehat{E' \times E'} = E'\times E' \ar[u]_-{(\hat{f},\id_{E'})} .
 }
\]
Hence $(p,q) \in E \times E'$ is mapped by $\varphi_{\Gamma} = (\hat{f}, \id_{E'}) \circ \varphi_{\Delta'} \circ (f,\id_{E'})$ as
\begin{align*}
(p,q) \mapsto (f(p), q) \mapsto (q-f(p), f(p) -q) \mapsto (\hat{f}(q)&-\hat{f}(f(p)), f(p)-q) .
\end{align*}
Since $\hat{f} (f(p)) = kp$ by $\hat{f} \circ f = \mu_k^E$, we have 
 $\varphi_{\Gamma}(p,q) = (\hat{f}(q)-kp, f(p)-q) \in E \times E'$ and 
this claim follows.
\end{proof}

By \autoref{claim_Gamma},
$t_{(p.q)}^* \calo_{X}(D) \otimes \calo_X(-D) $ corresponds to 
\[
a(-p,o_{E'}) +b (o_E, -q) + (\hat{f}(q)-kp, f(p)-q) =(\hat{f}(q)-(a+k)p, f(p)-(1+b)q) \in E \times E'.
\]
Hence $K(l) = \{ (p,q) \in E \times E' \, | \,  \hat{f}(q) =(a+k)p, f(p)=(1+b)q \}$
and (iv) follows.
\end{proof}

Now we can show \autoref{main_thm} for abelian surfaces:

\begin{prop}[{$=$ \autoref{main_thm} for $g=2$}]\label{prop_estimate_beta_surface}
Let $d \geq 1$ be an integer and set $m:=\lfloor  \sqrt{d} \rfloor$.
Let $(X,l)$ be a general polarized abelian surface of type $(1,d)$.
Then 
\begin{enumerate}
\item $1/\sqrt{d} \leq \beta(l) \leq 1/m$ holds.
\item $1/\sqrt{d} \leq \beta(l) \leq (m+1)/d < 1/m$ holds if $d \geq m^2+m+1$.
\end{enumerate}
\end{prop}

\begin{proof}
The lower bound $ 1/\sqrt[g]{d} \leq \beta(l) $ follows from \autoref{lem_divisor} (i).
Thus
it suffices to find an example $(X,l)$ of type $(1,d) $ which satisfies the upper bound in (1) or (2) 
by \autoref{thm_semicontinuity}.
We construct such examples as $(E  \times E', l_{a,b})$ by choosing suitable $k,a,b$ in \autoref{lem_example_surface}.

\vspace{2mm}
\ni
(1) Since (1) is true for $m =1$, we may assume $m \geq 2$.
Write $d= (m-1) q+r$ for integers $q,r$ with $1 \leq r \leq m-1$ and set
\[
k=q-r , \quad a=r, \quad b=m-1.
\]
We note that $k$ is positive since $q \geq m+1 > r$ by $d \geq m^2$,
and 
\[
a+ab +bk=r+ r (m-1) +(m-1)(q-r) = r+(m-1)q=d.
\]
Hence by \autoref{lem_example_surface},
$(E  \times E', l_{a,b})$ for these $k,a,b$ is of type $(1,d)$ and 
\begin{align*}
\beta(l_{a,b}) &\leq \max\left\{    \frac{1}{1+b} , \frac{1+b}{d}\right\} =  \max \left\{ \frac{1}{m} , \frac{m}{d} \right\} \leq \frac1m
\end{align*}
since $d \geq m^2$.

\vspace{2mm}
\ni
(2) Write $d= m q+r$ for integers $q,r$ with $1 \leq r \leq m$ and set
\[
k=q-r , \quad a=r, \quad b=m.
\]
We note that $k $ is positive since $q \geq m+1 >r$ by $d \geq m^2+m+1$,
and 
\[
a+ab +bk=r+ r m +m(q-r) = r+mq =d.
\]
Hence by \autoref{lem_example_surface},
$(E  \times E', l_{a,b})$ for these $k,a,b$ is of type $(1,d)$ and 
\begin{align*}
\beta(l_{a,b}) &\leq \max\left\{    \frac{1}{1+b} , \frac{1+b}{d}\right\} =  \max \left\{ \frac{1}{m+1} , \frac{m+1}{d} \right\} =\frac{m+1}{d}< \frac1m
\end{align*}
since $m^2 +m+1 \leq d < (m+1)^2$.
\end{proof}

As stated in Introduction,
a very ample line bundle $L$ of type $(1,d)$ with $d \geq 7$ on an abelian surface $X$ satisfies $(N_0)$, that is,
$L$ is projectively normal by \cite{Lazarsfeld_proj_normal}, \cite{MR2076454}.
Furthermore, in this case the homogeneous ideal of $X $ embedded by $|L|$ is generated by quadrics and cubics by \cite{MR3656291}.
For $d \geq 10$, a general polarized abelian surface of type $(1,d)$ satisfies $(N_1)$, that is,
the homogeneous ideal of $X $ embedded by $|L|$ is generated by quadrics by \cite{MR1602020}.

By \autoref{prop_estimate_beta_surface}, we can show \autoref{thm_intro} for abelian surfaces,
which partially recovers and generalizes the above results to higher syzygies.
Partial means that we cannot give explicit conditions for the generality of $(X,L)$ 
and the bound $d \geq 13$ for $(N_1) $ is larger than the bound $d \geq 10$ in \cite{MR1602020}.

\begin{cor}[{$=$ \autoref{thm_intro} for $g=2$}]\label{cor_N_p}
Let $p \geq -1$ be an integer and 
let $(X,L)$ be a general polarized abelian surface of type $(1,d)$.
Then 
\begin{enumerate}
\item $(N_p)$ holds for $L$ if $d \geq (p+2)^2+(p+2) +1 =p^2+5p+7$.
\item The homogeneous ideal of $X \subset \P^{d-1}$ embedded by $|L|$ is generated by quadrics and cubics if $d \geq 7$.
\end{enumerate}

\end{cor}

\begin{proof}
(1) follows from \autoref{thm_ Bpf_threshold} and \autoref{prop_estimate_beta_surface}.
For (2), it suffices to see that the Koszul cohomology group $K_{1,q}(X;L)=K_{1,q}(X,\calo_X;L)$ vanishes for any $q \geq 3$ (see \cite[p.\ 606]{MR2995182} for example),
which follows from $\beta(l) < 1/2 $ for $d \geq 7$ and \cite[Proposition 2.5]{Ito:2020aa}.
\end{proof}

\begin{rem}\label{rem_surface}
(1) Applying \cite[Theorem 1.2]{MR3923760}, which is a slight generalization of \cite[Theorem 1.1]{MR4009173}, to a general polarized abelian surface $(X,L)$ of type $(1,d)$,
we see that $(N_p)$ holds for $L$
if $d > 2(p+2)^2$.
\autoref{cor_N_p} improves the bound by  a factor of approximately $2$ for $p \gg 1$.\\
(2) Though the bound in \autoref{cor_N_p} (1) is a quadratic of $p$,
M.~Gross and S.~Popescu conjecture the following linear bound:
\cite[Conjecture (b)]{MR1602020} states  that  $L$ satisfies $(N_{\lfloor d/2 \rfloor -4} )$
for a general polarized abelian surface $(X,L)$ of type $(1,d)$ with $d \geq 10$.
Equivalently,
$(N_p) $ holds for such $L$ if $d  \geq 2p+8 \geq 10$.
\end{rem}

\section{On general polarized abelian varieties of type $(1,\dots,1,d)$}\label{sec_any_dim}

In this section,
we prove \autoref{main_thm}.
Although the argument becomes a little complicated,
the essential idea is the same as the surface case. 
In the following examples,
it is not difficult to bound $\beta(l)$.
To show that the type is of the form $(1,\dots,1,d)$,
we use \autoref{lem_example_surface}.

\vspace{1mm}
Let $g \geq 2, k_1,\dots,k_{g-1} \geq 1$ be integers and set $k_g:=1$.
Let $E_g$ be an elliptic curve
and take an elliptic curve $E_i$ and an isogeny 
$ f_i : E_i  \ra E_g$ for each $1 \leq i \leq g-1$ with $\ker f_i \simeq \Z/ k_i \Z$.
Then we have   
$
\hat{f}_i \circ f_i = \mu^{E_i}_{k_i} ,  f_i \circ \hat{f}_i =\mu^{E_g}_{k_i} 
$
for the dual isogeny $\hat{f}_i : E_g \ra E_i$.
Let $X=E_1 \times E_2 \times \dots \times E_g$ and 
let $F_i = \pr_i^* (o_{E_i})$,
where $\pr_i : X \rightarrow E_i$ is the projection to the $i$-th factor. 
A divisor $\Gamma$ on $X$ is defined by
\[
\Gamma = \left\{ \Big(  p_1,\dots,p_{g-1}, \sum_{i=1}^{g-1} f_i(p_i) \Big) \in X \, \Big| \, p_i \in E_i \text{ for }  1 \leq i \leq g-1 \right\}.
\]

The following proposition is a generalization of \autoref{lem_example_surface} (ii), (iii)
to higher dimensions:

\begin{prop}\label{prop_example_type(1,..., d)}
Under the above setting, 
let $l=l_{a,b} $ be the class of 
\[
D_{a,b}= a F_1 + \sum_{i=2}^{g-1} F_i + b F_g +\Gamma
\]
for integers $a,b \geq 0$ with $(a,b) \neq (0,0)$.
Set $N_i:=\sum_{j=i+1}^{g} k_j$ for $1 \leq i \leq g-1$ and $N_g:=0$.
Then
\begin{enumerate}
\item The type of $l$ is $(1,\dots,1,  d)$
for $d= a+ab N_1+b k_1$.
\item It holds that
\[
\frac{1}{1+b} \leq \beta(l) \leq \max\left\{  \max_{1 \leq i \leq g-1}  \frac{1+b N_{i+1}}{1+bN_i} , \frac{1+bN_1}{d}\right\}.
\]
\end{enumerate}
\end{prop}

\begin{ex}\label{rem_example}
Before proving \autoref{prop_example_type(1,..., d)},
we give an example.
If $g=3$, \autoref{prop_example_type(1,..., d)}
state that $l_{a,b}$ is of type $(1,1,a+ab+abk_2+bk_1)$ and
\begin{align}\label{eq_3fold}
\frac{1}{1+b} \leq \beta(l_{a,b}) \leq \max \left\{  \frac{1}{1+b}, \frac{1+b}{1+b+ bk_2}, \frac{1+b +b k_2}{a(1+b+bk_2)+bk_1} \right\}.
\end{align}
For example, if we take $a=1, b=3,k_1= 9,k_2=3$, then
$l_{1,3}$ is of type $(1,1,40)$ and $1/4 \leq \beta(l_{1,3}) \leq \max\{ 1/4,4/13,13/40\} =13/40 < 1/3$ holds.
\end{ex}

To show \autoref{prop_example_type(1,..., d)},
we prepare two lemmas.
We denote by $o_i \in E_i$ the origin of $E_i$.
By $E_i\stackrel{\sim}{\rightarrow} \widehat{E}_i : p_i \mapsto \calo_{E_i}(p_i -o_i)$, 
we may identify
$\widehat{X} =\widehat{E}_1 \times \dots \times \widehat{E}_g $ with $X$
as in the proof of \autoref{lem_example_surface}.
Under this identification,
the point in $\widehat{X}=X$ corresponding to $t_x^* \calo_X(F_i) \otimes \calo_X(-F_i)$ for $x=(p_1,\dots,p_g) \in X$ is 
$(o_1,\dots,o_{i-1},-p_i,o_{i+1},\dots,o_g)$ since 
\begin{align}\label{eq_F_i}
t_x^* \calo_X(F_i) \otimes \calo_X(-F_i) =\pr_i^* \calo_{E_i}((-p_i) - o_i).
\end{align}

The following is a generalization of \autoref{claim_Gamma}  to higher dimensions:

\begin{lem}\label{lem_gamma}
For $x=(p_1,\dots,p_g) \in X$,
$t_x^* \calo_X(\Gamma) \otimes \calo_X(-\Gamma)$ corresponds to
\[
( \hat{f}_1(A), \dots, \hat{f}_{g-1}(A),-A) \in X= \widehat{X},
\]
where $A=p_g -\sum_{i=1}^{g-1} f_i(p_i) \in E_g$.
\end{lem}

\begin{proof}
Set $E'_i=\{(q_1,\dots, q_g) \, | \, q_j=o_{g} \text{ for } j \neq i\} \subset Y:=E_g \times \dots \times E_g$
and consider a divisor
\[
\Delta' = \left\{ \Big(  q_1,\dots,q_{g-1}, \sum_{i=1}^{g-1} q_i \Big)  \, \Big| \, q_i \in E_g \text{ for }  1 \leq i \leq g-1 \right\} \subset Y=E_g \times \dots \times E_g.
\]
For $y=(q_1,\dots,q_g) \in  Y$,
we see that
\[
t_y^* \Delta' \cap E'_i =\{ A'\} \text{ for } 1 \leq i \leq g-1, \quad t_y^* \Delta' \cap E'_g =\{ -A'\} ,  \quad  \Delta' \cap E'_i =\{ o_{g}\} \text{ for } 1 \leq i \leq g
\]
for $A'=q_g-\sum_{i=1}^{g-1} q_i \in E_g$
under the natural identification $E'_i \simeq E_g$.
Hence $t_y^* \calo_Y(\Delta') \otimes \calo_Y(-\Delta')$ corresponds to $(A',\dots,A',-A') \in \widehat{Y}=Y$.

By definition,
$\Gamma$ is the pullback of $\Delta'$ by $(f_1,\dots,f_{g-1},\id_{E_g}) : X \ra Y$.
Hence
this lemma follows from the same argument as in \autoref{claim_Gamma} using
the commutative diagram
\[
\xymatrix@C=36pt
{
X\ar[r]^-{\varphi_{\Gamma}} \ar[d]_-{(f_1,\dots,f_{g-1},\id_{E_g})} & \widehat{X} = X\\
Y \ar[r]^-{\varphi_{\Delta'}} & \widehat{Y} = Y . \ar[u]_-{(\hat{f}_1,\dots,\hat{f}_{g-1}, \id_{E_g})} \\
 }
\]
In fact,
$A'$ for $(q_1,\dots,q_g) =(f_1(p_1),\dots,f_{g-1} (p_{g-1}), p_g )$ is nothing but $A=p_g -\sum_{i=1}^{g-1} f_i(p_i)  $.
Thus
$(p_1,\dots,p_g) \in X$ is mapped as
\[
 (p_1,\dots,p_g) \mapsto (f_1(p_1),\dots,f_{g-1} (p_{g-1}), p_g ) \mapsto (A,\dots,A,-A) \mapsto ( \hat{f}_1(A), \dots, \hat{f}_{g-1}(A),-A)
\]
by $ \varphi_{\Gamma} =(\hat{f}_1,\dots,\hat{f}_{g-1}, \id_{E_g}) \circ \varphi_{\Delta'} \circ (f_1,\dots,f_{g-1},\id_{E_g}) $.
\end{proof}

\begin{lem}\label{lem_K(l)}
Under the setting of \autoref{prop_example_type(1,..., d)},
it holds that
\[
K(l_{a,b}) \simeq  \left\{ (p_1,p_g ) \in E_1 \times E_g \, \Big| \,a p_1= -b \hat{f}_1(p_g) , f_1(p_1) =\left(1+bN_1 \right)p_g  \right\}.
\]
\end{lem}

\begin{proof}
By \ref{eq_F_i} and \autoref{lem_gamma},
$t_x^* \calo_{X}(D_{a,b}) \otimes \calo_X(-D_{a,b})$ for $x=(p_1,\dots,p_g)$ corresponds to the point
\begin{align}\label{eq_point}
(-a p_1+\hat{f}_1(A), -p_2+\hat{f}_2(A),\dots, -p_{g-1} +\hat{f}_{g-1}(A), -b p_g -A) \in X.
\end{align}
Hence $(p_1,\dots,p_g)$ is contained in $K(l_{a,b}) $ if and only if
\begin{align*}
& a p_1 =\hat{f}_1(A) ,   \quad p_i =\hat{f}_i(A ) \  \text{ for } \ 2 \leq i \leq g-1,  \quad  A =-b p_g\\
\lra \quad & a p_1=\hat{f}_1(-bp_g) , \quad p_i =\hat{f}_i(-b p_g) \  \text{ for }  \  2 \leq i \leq g-1, \quad A =-b p_g\\
\lra \quad & a p_1=-b\hat{f}_1(p_g) , \quad  p_i = - b \hat{f}_i( p_g) \ \text{ for } \ 2 \leq i \leq g-1,   \quad A =-b p_g.
\end{align*}
Under the condition $p_i = - b \hat{f}_i( p_g)  $ for $2 \leq i \leq g-1$, it holds that
\begin{align*}
b p_g+A= b p_g + p_g -\sum_{i=1}^{g-1} f_i(p_i)
&=b p_g + p_g - \sum_{i=2}^{g-1} f_i(-b \hat{f}_i(p_g)) -f_1 (p_1)\\
&=b p_g + p_g +b \sum_{i=2}^{g-1} f_i (\hat{f}_i(p_g)) -f_1 (p_1)\\
&=b p_g + p_g + b \sum_{i=2}^{g-1} k_i p_g -f_1(p_1)\\
&= \left(1+b\left(1+\sum_{i=2}^{g-1} k_i \right) \right)p_g -f_1(p_1)\\
&=\left(1+bN_1 \right)p_g -f_1(p_1),
\end{align*}
where the fourth equality follows from $f_i \circ \hat{f}_i =\mu^{E_g}_{k_i}$ 
and the last one follows from $k_g=1$ and $N_1=\sum_{i=2}^g k_i $.
Hence 
we have
\begin{align*}
K(l_{a,b}) &=\{ (p_1,\dots,p_g) \, | \, a p_1=-b\hat{f}_1(p_g), p_i = - b \hat{f}_i( p_g)   \text{ for }  2 \leq i \leq g-1, f_1(p_1) = \left(1+bN_1 \right)p_g    \}\\
&\simeq \{ (p_1,p_g) \, | \, a p_1=-b\hat{f}_1(p_g), f_1(p_1) = \left(1+bN_1 \right)p_g    \}.
\end{align*}
\end{proof}

\begin{proof}[Proof of \autoref{prop_example_type(1,..., d)}]
(1) Since $F_i, \Gamma \subset X$ are abelian subvarieties of codimension one,
$F_i^2=\Gamma^2=0$ as cycles on $X$.
By the definition of $\Gamma$,
it is easy to see that
\[
(F_1\cdots F_g) =(F_1\cdots F_{g-1}.\Gamma) =1, \quad (F_1 \cdots F_{i-1}.F_{i+1} \cdots F_g.\Gamma) =k_i \ \text{  for  } \ 1 \leq i \leq g-1. 
\]
Hence we have
\begin{align*}
\chi(l)=\frac{(l^g)}{g!} &= b (F_2\cdots F_g.\Gamma) + ab \sum_{i=2}^{g-1} (F_1\cdots F_{i-1}.F_{i+1}\cdots  F_g.\Gamma)  \\
&\hspace{60mm} + a (F_1\cdots F_{g-1}.\Gamma) + ab (F_1\cdots F_g) \\
&=b k_1 + ab \sum_{i=2}^{g-1} k_i + a  + ab\\
&= a+ab\left(1+\sum_{i=2}^{g-1} k_i  \right) +bk_1 =a+ab N_1 + b k_1=d .
\end{align*}
For simplicity, set $N:=N_1$. 
By \autoref{lem_K(l)},
we may identify $K(l)$ with
\begin{align*}
 \left\{ (p_1,p_g ) \in E_1 \times E_g \, \Big| \, a p_1= -b \hat{f}_1(p_g) , f_1(p_1) =(1+bN)p_g  \right\}.
\end{align*}
Consider a group
\begin{align*}
K'= \left\{ (p_1,p_g ) \in E_1 \times E_g \, \Big| \, a Np_1= -b N\hat{f}_1(p_g) , f_1(p_1) =(1+bN)p_g   \right\}.
\end{align*}
By definition, we have
\begin{align*}
NK' :=\{ (Np_1,Np_g) \, | \, (p_1,p_g) \in K' \} \subset K(l) \subset K'.
\end{align*}
Under the condition $  f_1(p_1) =(1+bN)p_g $,
\begin{align}\label{eq_claim_ample}
\begin{aligned}
a Np_1= -b N\hat{f}_1(p_g) \ &\Leftrightarrow \ a Np_1 +\hat{f}_1(f_1(p_1))= -b N\hat{f}_1(p_g) + \hat{f}_1((1+bN)p_g  )\\
&\Leftrightarrow  \ a Np_1 +k_1 p_1= -b N\hat{f}_1(p_g) + (1+bN) \hat{f}_1(p_g )\\
&\Leftrightarrow \  (a N+k_1) p_1=  \hat{f}_1(p_g ).
\end{aligned}
\end{align}
Hence we have
\[
K'= \left\{ (p_1,p_g ) \in E_1 \times E_g \, \Big| \,  \hat{f}_1(p_g) =(aN+k_1) p_1 ,  f_1(p_1) =(1+bN)p_g  \right\}.
\]
By \autoref{lem_example_surface} (iv), 
we have 
$K'=K(l')$ for $l':=l_{aN,bN}$ on $E \times E'$ with
 $E=E_1,E'=E_g, k=k_1 $ in \autoref{lem_example_surface}. 
By \autoref{lem_example_surface} (ii), $l'$ is of type $(1, aN+aNbN+ bNk_1) =(1,Nd)$ since $d= a+abN + bk_1$.
Hence 
$K'=K(l') \simeq (\Z / Nd \Z)^{ \oplus 2}$.
Since 
\[
K(l) \supset NK' \simeq (N\Z / Nd \Z)^{ \oplus 2} \simeq  (\Z / d \Z)^{ \oplus 2}
\]
and $|K(l)| = \chi(l)^2 =d^2$, we have $K(l) = N K' \simeq (\Z / d \Z)^{ \oplus 2}$.
Thus the type of $l$ is $(1,\dots,1, d)$. 

\vspace{1mm}
\noindent
(2) Set a flag $X=X_0 \supset X_1 \supset \dots \supset X_{g-1}$ of abelian subvarieties of $X$ as
\[
X_i=\{(p_1,\dots,p_g) \in X \, | \, p_j=o_j \text{ for } 1 \leq j \leq i\} = F_1 \cap \dots \cap F_{i} \subset X.
\]
Applying \autoref{lem_divisor} (ii) repeatedly,
we have
\[
\beta(l) \leq \max \left\{ \frac{1}{\chi(l|_{X_{g-1}})}, \frac{\chi(l|_{X_{g-1}})}{\chi(l|_{X_{g-2}})}, \dots, \frac{\chi(l|_{X_{1}})}{\chi(l|_{X_{0}})} \right\}
\]
since $\beta(l|_{X_{g-1}})   =1/\deg (l|_{X_{g-1}}) = 1/\chi(l|_{X_{g-1}})$ for the elliptic curve $X_{g-1}$.
Then the upper bound in (2) follows from
\[
\chi(l|_{X_i}) =\frac{(l|_{X_i}^{g-i})}{(g-i)!} = \frac{(F_1\cdots F_i.l^{g-i})}{(g-i)!} =1+bN_i
\]
for $1 \leq i \leq g-1$ and $\chi(l|_{X_0}) =\chi(l) =d$.
The lower bound $\beta(l) \geq 1/(1+b)$ holds by applying  \autoref{lem_divisor} (ii) to $Z= X_{g-1}$
since $ \beta( l|_{X_{g-1}}) = 1/\chi ( l|_{X_{g-1}}) =1/(1+b)$.
\end{proof}

Now we can show \autoref{main_thm}:

\begin{proof}[Proof of \autoref{main_thm}]
The lower bound $ 1/\sqrt[g]{d} \leq \beta(l) $ follows from \autoref{lem_divisor} (i).
If $g=1$,
this theorem holds since $ \beta(l) =1/\deg(l)$ in this case.
Thus
it suffices to find an example $(X,l)$ of type $(1,\dots,1,d) $ such that $ \beta (l) \leq 1/m $ or $\beta(l) < 1/m$ for $m=\lfloor \sqrt[g]{d} \rfloor$ and $g\geq 2$
by \autoref{thm_semicontinuity}.
We construct such examples as $(E_1 \times  \dots \times E_g, l_{a,b})$ for suitable $k_1,\dots,k_{g-1},a,b$.

We set $k_i= m^{g-i}$ for $2 \leq i \leq g-1$.
In this case, we have
\[
N_i=\sum_{j=i+1}^g k_j= \sum_{j=i+1}^g m^{g-j} =\sum_{n=0}^{g-i-1} m^n
\]
for $1 \leq i \leq g-1$ since $k_g=1$.
We take $k_1, a, b $ as follows.

\vspace{2mm}
\ni
(1) Write $d= (m-1) q+r$ for integers $q,r$ with $1 \leq r \leq m-1$ and set
\[
k_1=q-rN_1 , \quad a=r, \quad b=m-1.
\]
Since $d \geq m^g$, we have $q \geq m^{g-1 }  + \dots + m+ 1$.
Hence $k_1$ is positive by $N_1 =m^{g-2}+ \dots +m +1$ and $r \leq  m-1 $.
Furthermore,
\begin{align*}
1+b N_i &= 1+(m-1) \sum_{n=0}^{g-i-1} m^n = m^{g-i},\\
a+abN_1+bk_1 &= r+r(m-1)N_1 +(m-1) (q-rN_1) =r+(m-1)q=d.
\end{align*}
Hence by \autoref{prop_example_type(1,..., d)},
$(E_1\times \dots \times E_g, l_{a,b})$ for these $k_1,\dots,k_{g-1},a,b$ is of type $(1,\dots,1,d)$ and 
\begin{align*}
\beta(l_{a,b}) &\leq \max\left\{  \max_{1 \leq i \leq g-1}  \frac{1+b N_{i+1}}{1+bN_i} , \frac{1+bN_1}{d}\right\}\\
&=  \max \left\{ \frac{1}{m} , \frac{m}{m^2} , \dots ,\frac{m^{g-2}}{m^{g-1}} , \frac{m^{g-1}}{d} \right\} =\max  \left\{ \frac1m, \frac{m^{g-1}}{d}  \right\} \leq \frac1m
\end{align*}
since $1+bN_i=m^{g-i}$ and $d \geq m^g$.

\vspace{2mm}
\ni
(2) Write $d= m q+r$ for integers $q,r$ with $1 \leq r \leq m$ and set
\[
k_1=q-rN_1 , \quad a=r, \quad b=m.
\]
Since $d \geq m^g+ \dots +m+1$, we have $q \geq m^{g-1 }  + \dots + m+ 1$.
Hence $k_1$ is positive by $N_1 =m^{g-2}+ \dots +m +1$ and $r \leq  m $.
 Furthermore,
\begin{align*}
1+b N_i &= 1+m \sum_{n=0}^{g-i-1} m^n =\sum_{n=0}^{g-i} m^n,\\
a+abN_1+bk_1 &= r+rmN_1 +m (q-rN_1) =r+mq=d.
\end{align*}
Hence by \autoref{prop_example_type(1,..., d)},
$(E_1\times \dots \times E_g, l_{a,b})$ for these $k_1,\dots,k_{g-1},a,b$ is of type $(1,\dots,1,d)$ and 
\begin{align*}
\beta(l_{a,b}) &\leq \max\left\{  \max_{1 \leq i \leq g-1}  \frac{1+b N_{i+1}}{1+bN_i} , \frac{1+bN_1}{d}\right\} < \frac1m
\end{align*}
since
\[
\frac{1+bN_{i+1}}{1+bN_i} = \frac{m^{g-i-1} + \dots+m+1 }{ m^{g-i} + \dots +m+1} < \frac{1}{m}
\]
for $1 \leq i \leq g-1$ and
\[
\frac{1+b N_1}{d} = \frac{m^{g-1} + \dots+m+1 }{ d} \leq \frac{m^{g-1} + \dots+m+1 }{ m^{g} + \dots +m+1} < \frac{1}{m}
\]
by the assumption $ d \geq m^{g} + \dots +m+1$.
\end{proof}

\begin{proof}[Proof of \autoref{thm_intro}]
The proof is the same as that of \autoref{cor_N_p}:
The statement about $(N_p)$, in particular projective normality, follows from \autoref{thm_ Bpf_threshold} and \autoref{main_thm}.
The last statement about generators of homogenous ideal of $X$ 
follows from the vanishing $K_{1,q}(X;L)=0$ for any $q \geq 3$,
which follows from $\beta(l) < 1/2 $ for $d \geq 2^{g+1}-1$ and \cite[Proposition 2.5]{Ito:2020aa}.
\end{proof}

\begin{ex}\label{ex_3fold_low_deg}
If we choose $a,b,k_i$ carefully in \autoref{prop_example_type(1,..., d)},
we might obtain a better bound than that in \autoref{main_thm}.
For example, consider the case $g=3$ and $4 \leq d \leq 7$.
By \autoref{main_thm},
we know that $1/2 < 1/\sqrt[3]{d} \leq \beta(l) < 1$ for general $(X,l)$ of type $(1,1,d)$.
In fact,
we have the following bound by
taking $a=b=1$ and suitable $k_1,k_2$ in \ref{eq_3fold}:
\begin{itemize}
\item $d=4$ : $\beta(l) \leq 3/4$ by taking $(k_1,k_2)=(1,1)$.
\item $d=5$ : $\beta(l) \leq 2/3$ by taking $(k_1,k_2)=(2,1)$.
\item $d=6$ : $\beta(l) \leq 2/3$ by taking $(k_1,k_2)=(3,1)$ or $(2,2)$.
\item $d=7$ : $\beta(l) \leq 4/7$ by taking $(k_1,k_2)=(3,2)$. 
\end{itemize}
\end{ex}

\begin{rem}\label{rem_related_results}
We recall some related results and compare them to \autoref{thm_intro}.

\vspace{1mm}
\noindent
(1) For $p=0$, Iyer \cite[Theorem 1.2]{MR1974682} proves that 
an ample line bundle $L$ on a simple abelian variety of dimension $g$ is projectively normal if $\chi(L) > 2^g \cdot g!$.
Although \autoref{thm_intro} gives no explicit condition on the generality of $(X,L)$,
\cite{MR1974682} gives an explicit condition as $X$ is simple.
Furthermore,
\cite{MR1974682} has no assumption on the type of $L$.
On the other hand,
the bound $\chi(L) =d \geq 2^{g+1}-1$ in \autoref{thm_intro} is smaller than the bound in \cite{MR1974682}  
by a factor of approximately $g!/2$.

\vspace{1mm}
\noindent
(2) 
For $p \geq -1$,
R.~Lazarsfeld, G.~Pareschi and M.~Popa
\cite[Corollary B]{MR2833789} prove that 
$L$ satisfies $(N_p)$
if $\chi(L) > \frac{(4g)^g}{2g!} (p+2)^g$ and $(X,L)$ is very general.
\cite{MR2833789} also has no assumption on the type of $L$.
On the other hand,
the bound $\chi(L) \geq ((p+2)^{g+1}-1)/(p+1)$ in \autoref{thm_intro} is smaller than the bound in \cite{MR2833789} 
by a factor of approximately $ \frac{(4g)^g}{2g!} \cdot \frac{p+1}{p+2}$.

\vspace{1mm}
\noindent
(3) 
For $p \geq -1$,
the author \cite[Question 4.2]{MR3923760} asks if $L$ satisfies $(N_p)$
when $(L^{\dim Z}.Z) > ((p+2)\dim Z)^{\dim Z}$ holds for any abelian subvariety $Z \subset X$.
This question is answered affirmatively by \cite{MR3923760}, \cite{Ito:2020aa} for $g=2,3$ (see also \cite{MR4009173}, \cite{lozovanu:2018}).
In arbitrary dimension,
Z.~Jiang \cite[Theorem 1.5]{jiang2021} proves that $(N_p)$ holds for $L$
under the assumption  $(L^{\dim Z}.Z) > (2(p+2)\dim Z)^{\dim Z}$.

If this question has an affirmative answer for any $g \geq 1$,
$(N_p)$ holds for $L$
if $X$ is simple and $(L^g) > ((p+2)g)^g$,
equivalently $\chi(L) > \frac{g^g}{g!}  (p+2)^g$.
The bound $\chi(L)  \geq ((p+2)^{g+1}-1)/(p+1)$ in \autoref{thm_intro}
is smaller than the bound $\chi(L) > \frac{g^g}{g!}  (p+2)^g$
by a factor of approximately $\frac{g^g}{g!}  \cdot\frac{p+1}{p+2} $.

\vspace{1mm}
\noindent
(4)  Jiang also 
gives a numerical condition for a very general abelian variety $(X,L)$ to satisfy $(N_p)$
in \cite[Theorem 2.9]{jiang2021}.
As a special case, he \cite[Theorem 1.6]{jiang2021}  proves that $(N_p)$ holds for $L$ if
$(X,L)$ is very general of type $(1,\dots,1,d)$ and $d  > \frac{g^g}{g!}(p+2)^g$.
In particular,
\cite[Question 4.2]{MR3923760} has an affirmative answer for very general $(X,L)$ of type $(1,\dots,1,d)$.
In fact, the condition that $(X,L)$ is very general is explicit there, that is,
 \cite[Theorems 1.6,  2.9]{jiang2021} just
require the space of Hodge classes to be of dimension one in each degree.
Furthermore, 
\cite[Theorem 2.9]{jiang2021} treats $L$ of any type.

On the other hand,
the bound in \autoref{thm_intro}
is smaller than the bound in \cite[Theorem 1.6]{jiang2021}
by a factor of approximately $\frac{g^g}{g!}  \cdot\frac{p+1}{p+2} $ as in (3).
\end{rem}

\appendix

\section{Computation of $\beta(l)$ of general $(X,l)$ of type $(1,d)$ for some $d$}

Let $(X,l)$ be a general polarized abelian surface of type $(1,d)$.
By \autoref{prop_estimate_beta_surface} (1),
we have $\beta(l)=1/m$ if $d=m^2$ for some integer $m \geq 1$.
In the appendix,
we show that the upper bounds of $\beta(l)$ in \autoref{prop_estimate_beta_surface}
are sharp when $m $ is odd and $d$ is equal to $ m^2+m$
or $(m+1)^2-2$ or $(m+1)^2-1$.

\begin{lem}\label{lem_m(m+1)}
Let $m \geq 1$ be an odd integer. 
Let $(X,l)$ be a polarized abelian surface of type $(1,m^2+m)$.
Then $\beta(l) \geq 1/m$ holds.

In particular,
$\beta(l) = 1/m$ holds if $(X,l)$ is general.
\end{lem}

\begin{proof}
Let $\varphi_l : X \rightarrow \widehat{X}$ be the isogeny obtained by $p \mapsto t_p^* L \otimes L^{-1} \in \widehat{X}$.

\begin{claim}\label{claim_appendix}
There exists $\tilde{\sigma} \in X$ of order $2m$ such that the order of $\varphi_l(\tilde{\sigma}) \in \widehat{X}$ is $2m$ as well.
\end{claim}

\begin{proof}[Proof of \autoref{claim_appendix}]
Let $X_{2m}=\{ x \in X \, | \, 2m x=o_X \} \simeq (\Z/2m\Z)^{\oplus 4}$
and consider the exact sequence
\[
0 \ra  \ker \varphi_l  \cap X_{2m} \ra X_{2m} \ra \varphi_l(X_{2m}) \ra 0.
\]
Since $ \ker \varphi_l =K(l) \simeq (\Z/(m^2+m)\Z)^{\oplus 2}$,
the subgroup 
$ \ker \varphi_l  \cap X_{2m} \subset  \ker \varphi_l $ is generated by at most two elements.
By considering elementary divisors,
we see that there exists a basis $\sigma_1,\dots,\sigma_4$ of $X_{2m}$ as a free $\Z/2m \Z$-module 
such that the subgroup $ \ker \varphi_l \cap X_{2m} \subset X_{2m} $ is contained in  $\< \sigma_1,\sigma_2 \> \subset X_{2m}$.
Then $\tilde{\sigma}:=\sigma_3$ satisfies the condition in this claim.
\end{proof}

Take $\tilde{\sigma} \in X$ as in \autoref{claim_appendix} and set 
$\sigma:=\varphi_l(\tilde{\sigma}) \in \widehat{X}$.
Let $\pi : Y \rightarrow X$ be the dual isogeny of the quotient $\widehat{X} \rightarrow \widehat{X} / \langle \sigma \rangle$,
where $\Z/2m\Z \simeq  \langle \sigma \rangle \subset \widehat{X} $ is the subgroup generated by $\sigma$.
By a similar argument as the proof of \cite[Lemma 2.6]{MR1602020},
we can check that $K(\pi^* l) \simeq (\Z/2m\Z \oplus  \Z/(m^2+m)\Z)^{\oplus 2}$.
Since $m$ is odd by assumption,
$2m|m^2+m$ holds and hence $\pi^* l$ is of type $(2m,m^2+m)$.
Thus there exists a polarization $l'$  on $Y$ of type $(2,m+1)$ such that $\pi^*l =m l'$.
By \cite[Lemma 2.6]{Ito:2020aa},
\begin{align*}
\beta(l) < \tfrac{1}{m} \ &\Longleftrightarrow \ \cali_o \left\langle \tfrac{1}{m}  l \right\rangle \text{ is IT(0)}\\
&\Longleftrightarrow \ \pi^* \cali_{o} \left\langle \tfrac{1}{m}  \pi^*l \right\rangle = \cali_{\pi^{-1}(o)} \left\langle l' \right\rangle \text{ is IT(0)}\\
&\Longleftrightarrow  \ \cali_{\pi^{-1}(o)}   \otimes L' \text{ is IT(0)},
\end{align*}
where $L'$ is 
a line bundle on $Y$ representing $l'$ with characteristic $0$ with respect to some decomposition for $l'$.
The last condition is equivalent to
\begin{align*}
 h^0( \cali_{\pi^{-1}(o)+p}  \otimes L') = h^0( L') - \# \pi^{-1}(o) = 2(m+1) -2m =2
\end{align*}
for any $p \in Y$, where $ \pi^{-1}(o)+p $ is the parallel translation of $\pi^{-1}(o)$ by $p$.
Hence to show $\beta(l) \geq 1/m$,
it suffices to find a point $ p \in Y$ such that
$h^0( \cali_{\pi^{-1}(o)+p}   \otimes L') > 2$.

Since $L'$ is symmetric, $\Z/2\Z$ acts on $H^0(L')$.
Since $L'$ is of type $(2,m+1)$,
the dimension of the invariant part
$H^0(L')^{+} $ is $ h^0(L')/2 + 2=m+3$ by \cite[Corollary 4.6.6]{MR2062673}.
Take $\ep' \in Y$ such that $2 \ep' \in Y$ is a generator of $\ker \pi \simeq \Z/2m\Z$.
Then we have
\[
\pi^{-1}(o) - (2m-1)\ep' = \bigsqcup_{i=1}^{m} \{(2i-1)\ep', -(2i-1)\ep'\}.
\]
For a section $s \in H^0(L')^{+} $,
$s$ vanishes at $(2i-1)\ep' $ if and only if so does at $-(2i-1)\ep'$.
Hence $s \in H^0(L')^{+} $ is contained in $H^0( \cali_{\pi^{-1}(o) -(2m-1)\ep'} \otimes L' )$ if
$s$ vanishes at $m$ points $\ep',3\ep',\dots, (2m-1)\ep'$.
Thus $ h^0( \cali_{\pi^{-1}(o) -(2m-1)\ep'} \otimes L' )$ is greater than or equal to
\[
\dim H^0(L')^{+}  \cap H^0( \cali_{\pi^{-1}(o) -(2m-1)\ep'} \otimes L' ) \geq h^0(L')^{+} -m =3
\]
and $\beta(l) \geq 1/m$ follows.

The last statement follows from \autoref{prop_estimate_beta_surface} (1).
\end{proof}

\begin{rem}
In the first version of this paper,
the author wrote that it might be natural to guess $\beta(1,d) \geq \beta(1,d') $ for $d \leq d'$.
However,
this naive expectation is not true.
In fact,
a recent paper \cite{Rojas2021} investigates $\beta(l) $ for abelian surfaces using stability conditions,
and improves the bound in \autoref{prop_estimate_beta_surface}.
For example, $\beta(1,11) \leq 10/33$ holds by \cite[Theorem A]{Rojas2021}.
Hence we have $\beta(1,11) < 1/3 =\beta(1,3^2+3)=\beta(1,12) $.
\end{rem}

\begin{lem}\label{lem_m^2-1,2}
Let $m \geq 1$ be an odd integer. 
Let $(X,l)$ be a polarized abelian surface of type $(1,d)$
such that $d=(m+1)^2-1$ or $ d=(m+1)^2-2$.
Then $\beta(l) \geq (m+1)/d$ holds.

In particular,
$\beta(l) = (m+1)/d$ holds if $(X,l)$ is general.
\end{lem}

\begin{proof}
Since $(X,l)$ is of type $(1,d)$, there exist an isogeny $\pi : Y \rightarrow X$  and a principal polarization $\theta$ on $Y$ such that
$\pi^* l=d \theta$
and $\pi^{-1}(o)=\ker \pi \simeq \Z/d \Z$.
As in the proof of \autoref{lem_m(m+1)},
\begin{align*}
\beta(l) < \tfrac{m+1}{d} \ &\Longleftrightarrow \ \cali_o \left\langle \tfrac{m+1}{d}  l \right\rangle \text{ is IT(0)}\\
&\Longleftrightarrow \ \pi^* \cali_{o} \left\langle \tfrac{m+1}{d}  \pi^*l \right\rangle = \cali_{\pi^{-1}(o)} \left\langle (m+1) \theta \right\rangle \text{ is IT(0)}\\
&\Longleftrightarrow  \ \cali_{\pi^{-1}(o)}  \otimes \calo_Y((m+1)\Theta) \text{ is IT(0)},
\end{align*}
where 
$\calo_Y(\Theta)$ is a line bundle representing $\theta$ with characteristic $0$ with respect to some decomposition for $\theta$.
Hence to show $\beta(l) \geq (m+1)/d$,
it suffices to find $p \in Y$ such that 
\begin{align}\label{eq_2or3}
\begin{aligned}
 h^0( \cali_{\pi^{-1}(o)+p}  \otimes \calo_Y((m+1)\Theta)) &>  h^0( \calo_Y((m+1)\Theta) ) - \# \pi^{-1}(o) \\
 &= (m+1)^2  - d =  
  \begin{cases}
    1 & \text{ if } d=(m+1)^2-1 \\
    2  & \text{ if } d=(m+1)^2-2.
  \end{cases}
\end{aligned}
\end{align}

By \cite[Corollary 4.6.6]{MR2062673},
the dimension of the $\Z/2\Z$-invariant part $H^0(\calo_Y((m+1)\Theta)^+$ is $ (m+1)^2/2+2$
since $m$ is odd.
Let $\ep \in \pi^{-1}(o) $ be a generator of $\pi^{-1}(o) \simeq \Z/d \Z$.

\vspace{2mm}
\noindent
{\bf Case 1} :   $d=(m+1)^2-1$.
Since $m$ is odd, $d$ is odd and 
\[
\pi^{-1}(o) = \{o_Y\} \sqcup \bigsqcup_{i=1}^{\frac{d-1}{2}} \{i\ep, -i\ep\}.
\]
As in the proof of \autoref{lem_m(m+1)},
a section $ s \in H^0(\calo_Y((m+1)\Theta)^+$ is contained in $H^0( \cali_{\pi^{-1}(o)}  \otimes \calo_Y((m+1)\Theta) )$
if $s$ vanishes at $\frac{d+1}{2}=\frac{(m+1)^2}{2}$ points $o_Y,\ep,2\ep,\dots, \frac{d-1}{2} \ep$.
Thus
\begin{align*}
\dim H^0(\calo_Y((m+1)\Theta))^+ \cap H^0( \cali_{\pi^{-1}(o)}  \otimes \calo_Y((m+1)\Theta) )   \hspace{33mm} \\
\geq h^0(\calo_Y((m+1)\Theta))^+ - \frac{(m+1)^2}{2} =2.
\end{align*}
Hence \ref{eq_2or3} holds
for $p=o_Y$ and we have $\beta(l) \geq (m+1)/d$. 

\vspace{2mm}
\noindent
{\bf Case 2}: 
 $d=(m+1)^2-2$.
Since $m$ is odd, $d$ is even.
Take $\ep' \in Y$ such that $2\ep'=\ep$.
Then 
\[
\pi^{-1}(o) - (d-1)\ep' = \bigsqcup_{i=1}^{\frac{d}{2}} \{(2i-1)\ep', -(2i-1)\ep'\}.
\]
Hence $ s \in H^0(\calo_Y((m+1)\Theta)^+$ is contained in $H^0( \cali_{\pi^{-1}(o)-(d-1)\ep' }  \otimes \calo_Y((m+1)\Theta) )$
if $s$ vanishes at $\frac{d}{2}=\frac{(m+1)^2-2}{2}$ points $\ep', 3\ep',5\ep',\dots, (d-1)\ep'$.
Thus
\begin{align*}
\dim H^0(\calo_Y((m+1)\Theta))^+ \cap H^0( \cali_{\pi^{-1}(o)-(d-1)\ep' }  \otimes \calo_Y((m+1)\Theta) )  \hspace{20mm} \\
 \geq h^0(\calo_Y((m+1)\Theta))^+ - \frac{(m+1)^2-2}{2}=3.
\end{align*}
Hence \ref{eq_2or3} holds
for $p=- (d-1)\ep'$ and we have $\beta(l) \geq (m+1)/d$.

\vspace{2mm}
If $d   \geq  m^2+m+1$,
the last statement of this proposition follows from the first statement of this proposition and \autoref{prop_estimate_beta_surface} (2).
Since $d=(m+1)^2 -1 = m^2+2m$ or  $d=(m+1)^2 -2 = m^2+2m -1$,
the condition $d   \geq  m^2+m+1$ does not hold only when $m=1$ and $d=2$.
For $m=1$ and $d=2$, 
we have $ \beta(l)=1 = (m+1)/d$ by the case $p=-1$ of \autoref{thm_ Bpf_threshold} since $l$ is not basepoint free.
Thus the last statement of this lemma holds.
\end{proof}

\begin{ex}\label{ex_table}
By \autoref{prop_estimate_beta_surface} and \autoref{lem_m^2-1,2},
we have the following computation or estimation of $\beta(1,d)$ 
for small $d$.

\begin{table}[H]
  \centering
    \begin{tabular}{|c||c|c|c|c|c|c|c|c|c|c|c|c|c|c|c|c|} \hline
       $d$ & 1 & 2 & 3 & 4 & 5 & 6 & 7 & 8 & 9 & 10 & 11 & 12 & 13 & 14 & 15 & 16 \\ \hline 
      $\beta(1,d)$ & 1 & 1 & $\dfrac23$ & $\dfrac12$ & $\dfrac12$ & $\dfrac12$ & $\leq \dfrac37$  & $\leq \dfrac38$  
      & $\dfrac13$ & $\leq \dfrac13$  & $\leq \dfrac13$ & $\dfrac13$ & $\leq \dfrac{4}{13}$  & $\dfrac27$ & $\dfrac{4}{15}$ &  $\dfrac14$\rule[-4mm]{0mm}{11mm} \\ \hline
    \end{tabular}
\vspace{4mm}
\caption{$\beta(1,d)$ for $d \leq 16$}
\end{table}
For $d=1,2$, $\beta(1,d) =1$ follows from the case $p=-1$ of \autoref{thm_ Bpf_threshold} since $l$ is not basepoint free. 
For $d=3,14,15$,  $\beta(1,d) $ is computed by \autoref{lem_m^2-1,2}.
For $d=4,9,16$, $\beta(1,d) =1/\sqrt{d}$ follows from \autoref{prop_estimate_beta_surface} (1).
For $d=5,6$,
we have $\beta(1,d) \leq 1/2$ by \autoref{prop_estimate_beta_surface} (1).
Furthermore, $\beta(1,d) \geq 1/2$ follows from the case $p=0$ of \autoref{thm_ Bpf_threshold}
since $ l$ is not projectively normal for $d \leq 6$ by $\dim H^0(X,L^{\otimes 2}) > \dim \Sym^2 H^0(X,L)$.
For $d=12$, we have $\beta(1,d)=1/3$ by \autoref{lem_m(m+1)}.
For the rest $d$, the upper bounds of $\beta(1,d)$ follow from \autoref{prop_estimate_beta_surface}.
\end{ex}

\bibliographystyle{amsalpha}
\providecommand{\bysame}{\leavevmode\hbox to3em{\hrulefill}\thinspace}
\providecommand{\MR}{\relax\ifhmode\unskip\space\fi MR }
\providecommand{\MRhref}[2]{%
  \href{http://www.ams.org/mathscinet-getitem?mr=#1}{#2}
}
\providecommand{\href}[2]{#2}

\end{document}